\documentclass[12pt]{amsart}
\usepackage{amsfonts, amsmath, amsthm}

\usepackage{graphicx}
\usepackage{epstopdf}
\usepackage{psfrag}
\usepackage{color}

\newtheorem{pro}{Proposition}[section]
\newtheorem{thm}[pro]{Theorem}
\newtheorem{lem}[pro]{Lemma}

\newtheorem{cor}[pro]{Corollary}

\theoremstyle{definition}
\newtheorem{dfn}[pro]{Definition}

\theoremstyle{remark}

\newcommand{\bdy}{\partial}

\title{Surfaces that become isotopic after Dehn Filling}
\date{\today}
\address{Pitzer College}
\email{bachman@pitzer.edu}
\author{David Bachman}
\begin{document}

\address{Quest University Canada}
\email{rdt@questu.ca}
\author{Ryan Derby-Talbot}

\address{DePaul University}
\email{esedgwick@cdm.depaul.edu}
\author{Eric Sedgwick}

\begin{abstract}
We show that after generic filling along a torus boundary component of a 3-manifold, no two closed, 2-sided, essential surfaces become isotopic, and no closed, 2-sided, essential surface becomes inessential. That is, the set of essential surfaces (considered up to isotopy) survives unchanged in all suitably generic Dehn fillings. Furthermore, for all but finitely many non-generic fillings, we show that two essential surfaces can only become isotopic in a constrained way.
\end{abstract}

\maketitle

\section{Introduction}
\label{s:intro}

Let $M$ be a compact, orientable, irreducible 3-manifold, and $T$ a torus component of $\bdy M$. Let $M(\alpha)$ be the result of Dehn filling $M$ along a slope $\alpha$ on $T$. How does the set of 2-sided, closed, essential surfaces in $M$ relate to the set of 2-sided, closed, essential surfaces in $M(\alpha)$? For ``most'' slopes $\alpha$, we expect these sets to be the same.

While it is possible that $M(\alpha)$ contains an essential surface that is not (isotopic to) an essential surface in $M$, this yields a bounded essential surface in $M$ with boundary slope $\alpha$ on $T$.   Hatcher \cite{hatcher, floyd-oertel} demonstrated that the set of such slopes is finite, so this phenomenon is suitably restricted.

It may also be that a closed essential surface $F \subset M$ compresses in the filled manifold $M(\alpha)$.  This is constrained even more precisely by \cite{CGLS, Wu}:   If there is an incompressible annulus running between $F$ and $T$, then $F$ clearly compresses when filling along the slope of the annulus, call it $\beta$.  Moreover, if $F$ compresses when filling along any other slope $\alpha$, then $\alpha$ and $\beta$ intersect once and $F$ compresses in $M(\alpha')$ for precisely those $\alpha'$ for which $\alpha'$ and $\beta$ intersect once. If there is no such annulus, and $F$ compresses in both $M(\alpha)$ and $M(\alpha')$, then $\alpha$ and $\alpha'$ intersect once.  Thus, in the the non-annular case there are at most three compressing slopes in total.

Here we answer two remaining questions.  While we can avoid slopes for which any particular surface compresses, $M$ may contain infinitely many essential surfaces. Can we determine slopes $\alpha$ where {\it every} essential surface in $M$ remains incompressible in $M(\alpha)$? We answer this question via the following theorem:

\begin{thm} 
\label{t:MainCompressionTheorem}
Let $M$ be a compact, orientable, irreducible 3-manifold with a torus boundary component, $T$. Then there is a finite set of slopes $\Omega$ on $T$ such that for any slope $\alpha$ on $T$ and any closed, connected,  2-sided, essential surface $F$ in $M$, at least one of the following holds:
	\begin{enumerate}
		\item $F$ is incompressible in $M(\alpha)$.
		\item $\alpha$ intersects some slope $\omega \in \Omega$ once.
		\item $\alpha \in \Omega$.
        	\end{enumerate}
\end{thm}

The second conclusion is governed by the annular case of \cite{CGLS} cited above, and the surface in question will compress when filling along any slope that meets $\omega$ once.  

It is also possible that two non-isotopic closed essential surfaces become isotopic after Dehn filling. Can the set of slopes for which this occurs be restricted?  This second question is answered affirmatively by the following:

\begin{thm}
\label{t:MainIsotopyTheorem}
Let $M$ be a compact, orientable, irreducible 3-manifold with a torus boundary component, $T$. Then there is a finite set of slopes $\Omega$ on $T$ such that for any slope $\alpha$ on $T$ and any pair of non-isotopic, closed, connected, 2-sided surfaces $F$ and $G$ in $M$ where $F$ is essential in $M$, at least one of the following holds:
	\begin{enumerate}
		\item $F$ and $G$ are not isotopic in $M(\alpha)$.
		\item $\alpha$ intersects some slope $\omega \in \Omega$ once and there is a level isotopy in $M(\alpha)$ between $F$ and $G$.
		\item $\alpha \in \Omega$.
	\end{enumerate}
\end{thm}

The {\it level isotopy} of the second conclusion refers to a map $\gamma \colon F \times I \to M(\alpha)$ such that each component  of $\gamma ^{-1} (K)$ is contained in a level $F \times \{t\}$, where $K$ is the core of the solid torus attached to $M$ to form $M(\alpha)$.  In particular, in parallel with \cite{CGLS}, this implies that there is an annulus running between an essential surface $F$ and the boundary component $T$.  Moreover, if such a level isotopy exists, it will exist for any slope that meets the slope of the annulus once.

Together these theorems answer the main questions raised above. If $F$ is any essential surface in $M$, then for any slope $\alpha$ meeting every slope in $\Omega$ at least twice, $F$ cannot become compressible in $M(\alpha)$ by Theorem~\ref{t:MainCompressionTheorem} and cannot become isotopic to any other (essential or peripheral) surface in $M(\alpha)$ by Theorem~\ref{t:MainIsotopyTheorem}. We summarize this in the following corollary:

\begin{cor}
\label{t:SameSurfacesCorollary}
Let $M$ be a compact, orientable, irreducible 3-manifold with a torus boundary component, $T$. Then there is a finite set of slopes $\Omega$ on $T$ such that for any slope $\alpha$ on $T$ that meets every slope $\omega \in \Omega$ at least twice, there is a bijection between the set of closed, connected, 2-sided, essential surfaces of $M$ and the set of of closed, connected, 2-sided, essential surfaces of $M(\alpha)$, where the elements of each set are considered unique up to isotopy.
\end{cor}

Our proofs are conducted in the filled manifold $M(\alpha)$, where we consider separately the cases that an essential surface $F$ compresses to a surface $G$, in the case of Theorem \ref{t:MainCompressionTheorem}, or becomes isotopic to a different surface $G$, in the case of Theorem \ref{t:MainIsotopyTheorem}.  In Section~\ref{s:CompressingSequences} we construct a {\it compressing sequence}, a sequence of surfaces that starts with $F$, ends with $G$, and both encodes compressions and discretizes isotopies that pass through the core of the attached solid torus. Certain elements of a minimal compressing sequence behave like the thick levels for a knot in thin position. In particular, there are compressing disks on both sides of such a surface, and every compressing disk on one side  meets every compressing disk on the other (Corollary \ref{c:newsi}). By \cite{bss}, such a thick level must therefore be strongly irreducible and $\bdy$-strongly irreducible (and hence its boundary slope is in a finite set by \cite{bdts} and \cite{JacoSedgwick}), or its boundary slope meets the boundary slope of an incompressible, $\bdy$-incompressible surface (a finite set by \cite{hatcher}) at most once. It follows that $\alpha$ is therefore in a predetermined finite set of slopes $\Omega$, or meets one of the slopes in $\Omega$ once. In Section~\ref{s:isotopic} we show that in the latter case the core of the attached solid torus can be isotoped into every thick level, giving a level isotopy from $F$ to $G$.

\section{Compressing sequences}
\label{s:CompressingSequences}

For the remainder of this paper, $M$ will denote a compact, orientable, irreducible 3-manifold, $T$ a component of $\bdy M$ which is homeomorphic  to a torus, and $M(\alpha)$ the result of Dehn filling along $\alpha$, i.e.~attaching a solid torus to $M$ along $T$ so that a loop representing $\alpha$ bounds a disk. We will always denote the core of the attached solid torus as $K$.

\begin{dfn}
Let $F$ be a (possibly empty) embedded surface in $M(\alpha)$. If $F$ is empty, then we define the {\it width} $w(F)$ to be $(0,0)$. If $F$ is connected, then the {\it width} $w(F)$ is the pair, $(genus(F), |K \cap F|)$. If $F$ is disconnected, then its {\it width} is the ordered set of the widths of its components, where we include repeated pairs and the ordering is non-increasing. Comparisons are made lexicographically at all levels.
\end{dfn}

\begin{dfn}
Let $F$ be an embedded surface in $M(\alpha)$. A {\it compressing disk} for $F \cap M$ is an embedded disk $D \subset M$ such that $D \cap F=\bdy D$ is essential in $F \cap M$. Let $B$ be the closure of a neighborhood of $D$ in $M(\alpha) \setminus F$. Let $F'$ be the surface obtained from $F$ by removing $B \cap F$ and replacing it with the rest of $\bdy B$. Then we say $F'$ is obtained from $F$ by {\it compressing} along $D$ in $M$. A compressing disk $D$ for $F \cap M$ is {\it dishonest} if $\bdy D$ bounds a disk $D'$ in $F$. In this case the disk $D'$ is said to be the {\it witness} for $D$. If $D$ is not dishonest we say it is {\it honest}. See Figure \ref{fig:compression}.
\end{dfn}

\begin{figure}
\psfrag{D}{$D$}
\psfrag{p}{$\partial M$}
\psfrag{E}{$D'$}
\psfrag{F}{$F$}
  \centering
  \includegraphics[width=5in]{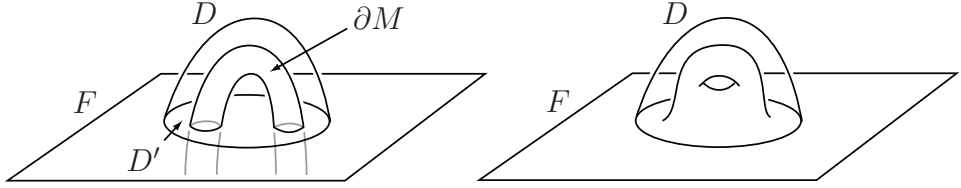}
  \caption{A dishonest compressing disk $D$ and its witness $D'$ (left) and
an honest compressing disk $D$ (right) for $F \cap M$.}
  \label{fig:compression}
\end{figure}

\begin{lem}
\label{l:widthchange}
Suppose $F'$ is obtained from a surface $F$ in $M(\alpha)$ by compressing along some compressing disk for $F \cap M$, followed by removing any resulting components that lie in a ball in $M(\alpha)$. Then $w(F')<w(F)$.
\end{lem}

\begin{proof}
Let $D$ be a compressing disk for $F \cap M$. Let $F'$ denote the surface obtained from $F$ by compressing along $D$, followed by removing any resulting components that lie in a ball in $M(\alpha)$. Suppose first that $D$ is honest, so that $\bdy D$ is essential in $F$. If, furthermore, $\bdy D$ is non-separating, then the genus of $F'$ is less than the genus of $F$, and hence the width is also less. If, on the other hand, $\bdy D$ is separating in $F$ then $F'$ is disconnected, and both components have smaller genus. Hence, again width has decreased.

Now suppose $D$ is dishonest, and $D'$ is its witness. Thus, $D' \cap K \ne \emptyset$. The result of compressing along $D$ produces a sphere in $M(\alpha)$ and a surface that has the same genus as $F$, but meets $K$ fewer times, and thus has smaller width than $F$. This latter surface is precisely $F'$.
\end{proof}

\begin{dfn}
A {\it compressing sequence} is a sequence $\{F_i\}$ of (possibly empty) 2-sided, embedded surfaces in $M(\alpha)$ such that for each $i$, either $F_i$ or $F_{i+1}$ is obtained from the other by compressing along a compressing disk in $M$, followed by discarding any resulting components that lie in a ball in $M(\alpha)$.
\end{dfn}

Note that by Lemma \ref{l:widthchange}, the widths of consecutive elements of a compressing sequence must be different. This motivates the next definition.

\begin{dfn}
Let $\{F_i\}$ be a compressing sequence. An element $F_i$ is said to be a {\it thick level} of the sequence if $w(F_i)>w(F_{i \pm 1})$.
\end{dfn}

\begin{dfn}
Let $F$ be a 2-sided, embedded surface in $M(\alpha)$. The {\it disk complex} of $F \cap M$ is the complex whose vertices correspond to isotopy classes of compressing disks for $F \cap M$. Two such vertices are connected by a 1-simplex if there are representatives of the corresponding isotopy classes that are disjoint.
\end{dfn}

Note that the endpoints of a 1-simplex of the disk complex of a surface may represent disks on the same side of the surface, or on opposite sides.

\begin{dfn}
\label{d:angle}
Let $\{F_i\}$ be a compressing sequence and $F_i$ a thick level. Then $F_{i-1}$ is obtained from $F_i$ by compressing along a disk $D \subset M$, and $F_{i+1}$ is obtained from $F_i$ by compressing along a disk $E \subset M$. If $D$ and $E$ can be connected by a path in the disk complex of $F_i \cap M$, then we say the {\it angle} $\angle(F_i)$ is the minimal length of such a path. If $D$ and $E$ cannot be connected by such a path, then we say $\angle(F_i)=\infty$.
\end{dfn}

\begin{dfn}
Let $\{F_i\}$ be a compressing sequence. Then the {\it size} of the entire sequence is the ordered set of pairs
\[\{(w(F_i), \angle(F_i)) \ | \  F_i \mbox{ is a thick level}\},\]
where repeated pairs are included, and the ordering is non-increasing. Two such sets are compared lexicographically.
\end{dfn}

\begin{dfn}
A compressing sequence $\{F_i\}_{i=0}^n$ is {\it minimal} if its size is smallest among all sequences from $F_0$ to $F_n$.
\end{dfn}

\begin{lem}
\label{l:InfiniteDistance}
Let $\{F_i\}$ be a minimal compressing sequence, and let $F_i$ be a thick level. Then $\angle(F_i)=\infty$.
\end{lem}

\begin{proof}
The proof is by induction on $\angle(F_i)$. Let $D$ and $E$ be as in Definition \ref{d:angle}. If $\angle(F_i)=0$ then $D=E$. In this case $F_{i-1}=F_{i+1}$, and we can obtain a smaller compressing sequence by removing the subsequence $\{F_i, F_{i+1}\}$.

If $\angle(F_i)=1$ then $D \cap E=\emptyset$. There are now two subcases:

\medskip

\noindent {\it Case 1.} Either $D$ or $E$ is honest, or they are both dishonest and have disjoint witnesses. Up to symmetry, there are now two further subcases.

{\it Subcase 1.1.} $E$ is a compressing disk for $F_{i-1}$ and $D$ is a compressing disk for $F_{i+1}$. Then $F_-$, the surface obtained from $F_{i-1}$ by compressing along $E$, can also be obtained from $F_{i+1}$ by compressing along $D$. In this case we alter the original compressing sequence by replacing $F_i$ with $F_-$. The new sequence has at most two new thick levels (namely, $F_{i \pm 1}$), but they both have smaller width than $F_i$.

{\it Subcase 1.2.} $E$ is not a compressing disk for $F_{i-1}$. Then $\bdy E$ bounds a disk $E'$ on $F_{i-1}$. Since $\bdy E$ did not bound a disk on $F_i$, it follows that $E'$ contains at least one of the two copies of $D$ that are subdisks of $F_{i-1}$. If $E'$ only contains one of these copies, then $\bdy E$ was parallel to $\bdy D$ on $F_{i-1}$, and thus $D \cup E$ is isotopic to a sphere meeting $F_{i-1}$ in a single circle. Thus, compressing $F_i$ along $D$ produces the surface $F_{i-1}$, together with a discarded component that lies in a ball. Similarly, compressing $F_i$ along $E$ produces the surface $F_{i+1}$, together with a discarded surface in a ball. By the irreducibility of $M(\alpha)$, it follows that $F_{i-1}$ is isotopic to $F_{i+1}$, and thus we can obtain a smaller compressing sequence by removing the subsequence $\{F_i,F_{i+1}\}$. 

If $E' \subset F_{i-1}$ contains both copies of $D$, and $E$ and $D$ were on opposite sides of $F_i$, then $E'$ can be pushed off of $F_{i-1}$ (to the $D$ side) so that $E \cup E'$ is a sphere meeting $F_i$ in a single loop. In this case again, the surface $F_{i-1}$ is the same as $F_{i+1}$, and thus there is a smaller compressing sequence. 

Finally, if $E'$ contains both copies of $D$, and $E$ and $D$ were on the same side of $F_i$, then $F_{i-1}$ is isotopic to the surface obtained from $F_{i+1}$ by compressing along $D$, following by discarding a trivial sphere. Thus, removing $F_i$ from the original compressing sequence produces a smaller one.

\medskip

\noindent {\it Case 2.} $D$ and $E$ are dishonest, and the witness for one contains the witness for the other. Let $D'$ and $E'$ denote the two witnesses, and assume $E' \subset D'$. Let $A$ denote the annulus that is the closure of $D' \setminus E'$. If $A \cap K =\emptyset$, then the sphere $D \cup A \cup E$ is essential in $M$, violating irreducibility. We conclude $A \cap K \ne \emptyset$. It follows that $D$ is a compressing disk for $F_{i+1} \cap M$ in $M$, and compressing along this disk produces the surface $F_{i-1}$. Hence, we may produce a new compressing sequence by simply removing the surface $F_i$ from the original one. The surface $F_{i+1}$ may now be a thick level of the new compressing sequence, but its width is smaller than the width of $F_i$. Hence, we have produced a smaller compressing sequence.

\medskip

If $\angle(F_i) >1$ then let $\{C_i\}_{i=0}^n$ be a sequence of compressing disks for $F_i \cap M$ in $M$ such that $\{D, C_0,..., C_n,E\}$ is a minimal length path in its disk complex. Let $F_-$ be the surface obtained from $F_i$ by compressing along $C_0$. Now consider the compressing sequence $\{...F_{i-1}, F_i, F_-, F_i, F_{i+1},...\}$. Both occurrences of $F_i$ in this new sequence are thick levels with the same width as before, but in both cases we have reduced the angle. Hence we have produced a smaller compressing sequence.
\end{proof}

\begin{cor}
\label{c:OppositeSides}
Let $\{F_i\}$ be a minimal compressing sequence and $F_i$ a thick level. Let $D$ and $E$ be as in Definition \ref{d:angle}. Then $D$ and $E$ are on opposite sides.
\end{cor}

\begin{proof}
By Lemma \ref{l:InfiniteDistance}, $\angle(F_i)=\infty$. Hence, by definition there is no path in the disk complex of $F_i \cap M$ from $D$ to $E$. But if $D$ and $E$ are on the same side $X$ of $F_i$, then we can construct such a path by successively $\bdy$-compressing $D$ along subdisks of $E$ in $X$.
\end{proof}

\begin{cor}
\label{c:stronglyirreducible}
Let $\{F_i\}$ be a minimal compressing sequence and $F_i$ a thick level. Then every pair of compressing disks for $F_i \cap M$ on opposite sides must intersect.
\end{cor}

\begin{proof}
Let $D$ and $E$ be as in Definition \ref{d:angle}. By Corollary \ref{c:OppositeSides}, $D$ and $E$ are on opposite sides of $F_i$. Let $B$ and $C$ be disjoint compressing disks for $F_i \cap M$ on opposite sides, where $B$ and $D$ are on the same side. As in the proof of Corollary \ref{c:OppositeSides}, there must then be a path in the disk complex from $D$ to $B$ in the disk complex of $F_i \cap M$. Similarly, the disks $C$ and $E$ are on the same side, so there is a path between them in the disk complex. Finally, as $B$ and $C$ are disjoint, there is an edge between them in the disk complex. Putting the two paths together with this edge then forms a path from $D$ to $E$ in the disk complex. But this contradicts Lemma \ref{l:InfiniteDistance}, which asserts that there is no such path.
\end{proof}

\begin{dfn}
A properly embedded surface $F$ in a 3-manifold $M$ is {\it strongly irreducible} if there is at least one compressing disk on each side, and every pair of compressing disks on opposite sides intersects. 
\end{dfn}

\begin{cor}
\label{c:newsi}
Let $\{F_i\}$ be a minimal compressing sequence and $F_i$ a thick level. Then $F_i \cap M$ is strongly irreducible. 
\end{cor}

\begin{proof}
By Corollary \ref{c:OppositeSides} there exist compressing disks on each side of $F_i$, and by Corollary \ref{c:stronglyirreducible} any such pair of disks on opposite sides must intersect.
\end{proof}

\begin{dfn}
A {\it $\bdy$-compressing disk} for a properly embedded surface $F \subset M$ is a disk $D$ such that $\bdy D=\alpha \cup \beta$, where $\alpha = D \cap F$ is an essential arc on $F$ and $D \cap \bdy M =\beta$. We say $F$ is {\it $\bdy$-compressible} if it has a $\bdy$-compressing disk, and is {\it $\bdy$-incompressible} otherwise. We say $F$ is {\it $\bdy$-strongly irreducible} if there is at least one compressing or $\bdy$-compressing disk on each side of $F$, and any pair of compressing or $\bdy$-compressing disks on opposite sides intersects. 
\end{dfn}

The following terminology is due to Jesse Johnson.
\begin{dfn}
Let $F$ be a surface in $M(\alpha)$. Let $D$ be an embedded disk such that $\bdy D =\alpha \cup \beta$, $D \cap F=\alpha$, and $D \cap K=\beta$. Let $B$ be the closure of a neighborhood of $D$ in $M(\alpha) \setminus F$. Let $F'$ be the surface obtained from $F$ by removing $B \cap F$ and replacing it with the rest of $\bdy B$. Then we say $F'$ is obtained from $F$ by {\it bridge compressing} along $D$. The disk $D$ is a {\it bridge compressing disk} for $F$. The disk $\overline D$ that is the frontier of $B$ in $M(\alpha) \setminus F$ is said to be the compressing disk for $F \cap M$ that is {\it associated} to the bridge compressing disk $D$. See Figure \ref{fig:Bridge_compressing_disk}.
\end{dfn}

\psfrag{R}{$\overline{D}$}
\psfrag{D}{$D$}
\psfrag{K}{$K$}

\begin{figure}
   \centering
   \includegraphics[width=3.5in]{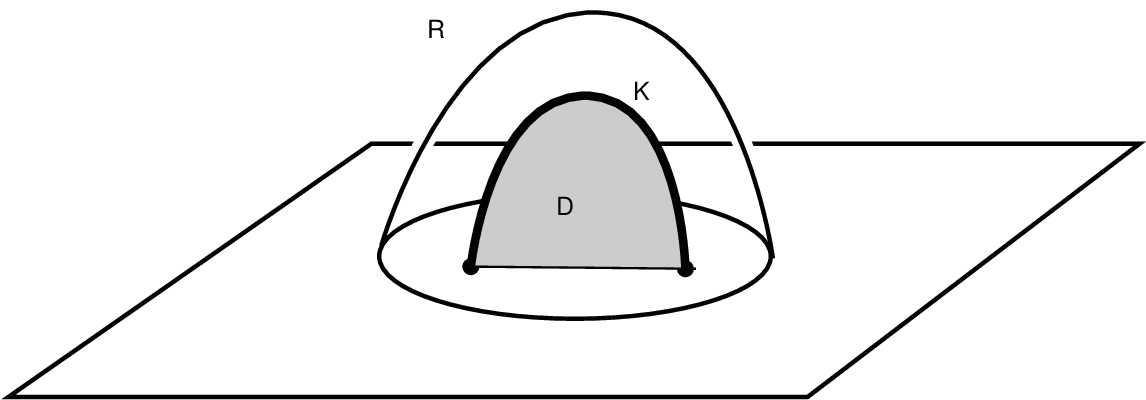}
   \caption{A bridge compressing disk and its associated compressing disk.}
   \label{fig:Bridge_compressing_disk}
\end{figure}

Note that a bridge compressing disk for $F \subset M(\alpha)$ meets $M$ in a $\bdy$-compressing disk for $F \cap M$. However, not every $\bdy$-compressing disk for $F \cap M$ is the intersection of a bridge compressing disk with $M$.


\begin{lem}
\label{l:StronglyIrreducible}
Let $\{F_i\}$ be a minimal compressing sequence and $F_i$ a thick level. Then there are two cases:
	\begin{enumerate}
		\item $F_i \cap M$ is strongly irreducible and $\bdy$-strongly irreducible.
		\item $\bdy(F_i \cap M)$ intersects the slope bounding an incompressible and $\bdy$-incompressible surface in $M$ at most once.
	\end{enumerate}
Furthermore, when $\bdy(F_i \cap M)$ intersects the slope bounding an incompressible and $\bdy$-incompressible surface in $M$ exactly once, there are bridge compressing disks on opposite sides of $F_i$ that meet in two points of $K$.
\end{lem}

\begin{proof}
By Corollary \ref{c:newsi}, a thick level $F_i$ meets $M$ in a surface that is strongly irreducible. By Lemma 4.8 of \cite{bss}, such a surface is either $\bdy$-strongly irreducible, or has boundary that is a distance of at most one from the boundary of a surface $S$ that is both incompressible and $\bdy$-incompressible. The last paragraph in this proof implies that when the boundary of $F_i \cap M$ is at a distance of exactly one from $S$, there are $\bdy$-compressions on opposite sides of $F_i \cap M$ that are the intersections of bridge compressing disks with $M$, on opposite sides of $F_i$, that meet in two points of $K$.
\end{proof}

\section{Surfaces after Dehn filling}
\label{s:isotopic}

\begin{lem}
\label{l:DishonestSequence}
Suppose $F$ and $G$ are isotopic surfaces in $M(\alpha)$ that are transverse to $K$. Then there exists a compressing sequence $\{F_i\}_{i=0}^n$ such that
	\begin{enumerate}
		\item $F_0=F$,
		\item $F_n=G$, and
		\item for each $i$, either $F_i$ or $F_{i+1}$ is obtained from the other by a dishonest compression.
	\end{enumerate}
\end{lem}

We call a sequence given by the conclusion of the lemma a {\it dishonest compressing sequence.}

\begin{proof}
Let $\gamma \colon F \times I \to M(\alpha)$ be an isotopy from $F$ to $G$ which is in general position with respect to $K$. Let $F(t)=\gamma(F,t)$. Let $\{t_i'\}$ denote the values of $t$ for which $F(t)$ is not transverse to $K$. For each $i$, choose $t_i \in (t'_{i-1},t'_i)$. Then for each $i$, either $F(t_i)$ or $F(t_{i+1})$ is obtained from the other by bridge compressing along some disk $D$. It follows that one of these surfaces can be obtained from the other by compressing along the associated disk $\overline D$, and throwing away a 2-sphere in $M(\alpha)$. Note that the disk $\overline D$ is dishonest, and thus $\{F(t_i)\}$ is a dishonest compressing sequence from $F$ to $G$.
\end{proof}

\begin{lem}
\label{l:CompressionImpliesThickLevel}
Suppose $F$ is a closed, connected,  essential, 2-sided surface in $M$, and $G$ is a (possibly empty) surface in $M(\alpha)$ obtained by a compression of $F$ followed by discarding sphere components in $M(\alpha)$. Then there exists a compressing sequence from $F$ to $G$, any minimal such sequence has a thick level, and the first such thick level meets $K$.
\end{lem}

\begin{proof}
If $F$ is compressible in $M(\alpha)$, then by shrinking the compressing disk off of $K$, it is isotopic in $M(\alpha)$ to a surface $G'$ which can be compressed in $M$ by an honest compressing disk. (The surface $G'$ may or may not meet $K$.) By Lemma \ref{l:DishonestSequence} there exists a dishonest compressing sequence from $F$ to $G'$. Note that the width of all surfaces in this sequence is of the form $(g,x_i)$, where $g$ is the genus of $F$. We may add one more element to this compressing sequence, namely the (possibly empty) surface $G''$ obtained from $G'$ by compressing in $M$ and possibly discarding a sphere in $M(\alpha)$, to obtain a compressing sequence from $F$ to $G''$. By Lemma \ref{l:DishonestSequence} there is now a dishonest compressing sequence from $G''$ to $G$. Putting this all together, we obtain a compressing sequence from $F$ to $G$ in which the genus of every element is at most $g$.

Now let $\{F_i\}_{i=0}^n$ be a minimal compressing sequence from $F$ to $G$, so that $F_0=F$ and $F_n=G$. Since the size of this sequence is at most the size of the sequence constructed above, the genus of each component of $F_i$ is at most $g$, for every $i$.

As $F \cap K=\emptyset$, $w(F)=(g,0)$. Furthermore, as $G$ is obtained from $F$ by a compression, $w(F_n)=w(G)<w(F)$.

Since $F$ is incompressible in $M$, it follows that $F_0=F$ must be obtained from $F_1$ by a compression. Thus, $w(F_1)>w(F_0)$ and therefore $w(F_1)>w(F_n)$. It follows that for some $i$, $w(F_i)>w(F_{i \pm 1})$, and thus $F_i$ is a thick level.

Since the genus of each component of $F_1$ is at most $g$ and $F_0$ is connected, it must be the case that the compression which results in the surface $F_0$ was dishonest. Thus, $F_1 \cap K \ne \emptyset$. It follows that the first thick level meets $K$.
\end{proof}

We are now prepared to prove Theorem \ref{t:MainCompressionTheorem}. Recall the statement:

\bigskip
\noindent {\bf Theorem 1.1.} {\it Let $M$ be a compact, orientable, irreducible 3-manifold with a torus boundary component, $T$. Then there is a finite set of slopes $\Omega$ on $T$ such that for any slope $\alpha$ on $T$ and any closed, connected, 2-sided, essential surface $F$ in $M$, at least one of the following holds:
	\begin{enumerate}
		\item $F$ is incompressible in $M(\alpha)$.
		\item $\alpha$ intersects some slope $\omega \in \Omega$ once.
		\item $\alpha \in \Omega$.
        	\end{enumerate}}

\begin{proof}
Let $\Omega$ be the set of slopes bounding surfaces that are incompressible and $\bdy$-incompressible or strongly irreducible and $\bdy$-strongly irreducible. By \cite{hatcher} the set of slopes bounding incompressible and $\bdy$-incompressible surfaces is finite. By \cite{bdts}, every strongly irreducible and $\bdy$-strongly irreducible surface can be made almost normal with respect to a triangulation of $M$ that has one vertex on $T$. By \cite{JacoSedgwick}, the set of slopes bounding such surfaces is finite. Hence, $\Omega$ is a finite set. 

Let $F$ be a closed, connected, 2-sided, essential surface in $M$. Suppose $G$ is a surface that is obtained from $F$ by a compression in $M(\alpha)$ followed by discarding any resulting sphere components. It follows from Lemma \ref{l:CompressionImpliesThickLevel} that there exists a minimal compressing sequence from $F$ to $G$ that has a thick level, and the first such thick level $F_i$ meets $K$. Then by Lemma \ref{l:StronglyIrreducible} there are two cases:
	\begin{enumerate}
		\item $F_i \cap M$ is strongly irreducible and $\bdy$-strongly irreducible, and hence $\bdy (F_i \cap M) \in \Omega$.
		\item $\bdy(F_i \cap M)$ intersects the slope bounding an incompressible and $\bdy$-incompressible surface in $M$ at most once. Hence, $\bdy (F_i \cap M)$ intersects some slope in $\Omega$ at most once.
	\end{enumerate}



\end{proof}

\begin{lem}
\label{l:MinimalIsDishonest}
Suppose $F$ and $G$ are non-isotopic, closed, connected surfaces in $M$ that are isotopic in $M(\alpha)$, and that $F$ is essential in $M$. Then either $M(\alpha)$ is reducible, or there exists a compressing sequence from $F$ to $G$ and any minimal such compressing sequence is dishonest and has a thick level.
\end{lem}

\begin{proof}
By Lemma \ref{l:DishonestSequence} there exists a dishonest sequence $\{F_i'\}$ from $F$ to $G$. Note that for each $i$, $w(F'_i)=(g,x_i)$, where $g$ is the genus of $F$. Now let $\{F_i\}$ denote a minimal compressing sequence from $F$ to $G$. In particular, it follows that the size of this sequence is at most the size of the sequence $\{F'_i\}$.  Thus, for all $i$, $w(F_i) \le \max \{w(F'_j)\}=(g,x)$ for some $x$. In particular, for each $i$ the genus of $F_i$ is at most $g$.

Assume that for each $i\le n$, either $F_i$ or $F_{i-1}$ is obtained from the other by a dishonest compression. If $M(\alpha)$ is irreducible, then such a compression can be realized by an isotopy in $M(\alpha)$. We conclude $F_n$ is isotopic in $M(\alpha)$ to $F$. Thus $F_n$ is incompressible in $M(\alpha)$. It follows that $F_{n+1}$ cannot be obtained from $F_n$ by an honest compression. Since $F_n$ is connected and the genus of each component of $F_{n+1}$ is at most $g$, we also conclude that $F_n$ cannot be obtained from $F_{n+1}$ by an honest compression. We conclude that either $F_n$ or $F_{n+1}$ is obtained from the other by a dishonest compression. By induction the entire compressing sequence $\{F_i\}$ is dishonest.

What remains is to show that $\{F_i\}$ has a thick level. Note that as $F$ and $G$ miss $K$, $w(F)=w(G)=(g,0)$. Since $F$ is incompressible in $M$, it follows that $F_0=F$ must be obtained from $F_1$ by a compression. Thus, $w(F_1)$ is larger than the widths of the first and last element of $\{F_i\}$.  It follows that for some $i$, $w(F_i)>w(F_{i \pm 1})$, and thus $F_i$ is a thick level.
\end{proof}

We are now prepared to prove Theorem \ref{t:MainIsotopyTheorem}. Recall the statement:

\bigskip
\noindent {\bf Theorem 1.2.} {\it Let $M$ be a compact, orientable, irreducible 3-manifold with a torus boundary component, $T$. Then there is a finite set of slopes $\Omega$ on $T$ such that for any slope $\alpha$ on $T$ and any pair of non-isotopic, closed, connected, 2-sided essential surfaces $F$ and $G$ in $M$ where $F$ is essential in $M$, at least one of the following holds:
	\begin{enumerate}
		\item $F$ and $G$ are not isotopic in $M(\alpha)$.
		\item $\alpha$ intersects some slope $\omega \in \Omega$ once and there is a level isotopy in $M(\alpha)$ between $F$ and $G$.
		\item $\alpha \in \Omega$.
	\end{enumerate}}

\begin{proof}
Let $\Omega$ be the set of slopes bounding surfaces that are incompressible and $\bdy$-incompressible or strongly irreducible and $\bdy$-strongly irreducible. As in the proof of Theorem \ref{t:MainCompressionTheorem}, the set $\Omega$ is finite. 

If $M(\alpha)$ is reducible then there is an essential planar surface in $M$ whose boundary slope on $T$ is $\alpha$, and thus $\alpha \in \Omega$. Thus we assume $M(\alpha)$ is irreducible. 

Let $F$ and $G$ be non-isotopic, closed, connected, 2-sided, essential surfaces in $M$ that are isotopic in $M(\alpha)$. By Lemma \ref{l:MinimalIsDishonest} there exists a compressing sequence from $F$ to $G$, and any minimal such sequence $\{F_i\}$ is dishonest and has a thick level. It follows that all thick levels meet $K$.

By Lemma \ref{l:StronglyIrreducible}, either 
\begin{enumerate}
	\item some thick level meets $M$ in a strongly irreducible and $\bdy$-strongly irreducible surface, or 
	\item some thick level meets $\bdy M$ in a slope that intersects the slope bounding an incompressible, $\bdy$-incompressible surface at most once. 
\end{enumerate}

In the first case, $\alpha \in \Omega$. In the second case, either $\alpha \in \Omega$ or by Lemma \ref{l:StronglyIrreducible}, $\alpha$ meets some slope in $\Omega$ once and there are bridge compressing disks on opposite sides of each thick level that meet in two points of $K$. To complete the proof of Theorem \ref{t:MainIsotopyTheorem}, we must now show in the latter case that there is a level isotopy from $F$ to $G$.

Recall that when $M(\alpha)$ is irreducible then a dishonest compressing sequence gives rise to an isotopy between $F$ and $G$, with each $F_i$ corresponding to an intermediate level of the isotopy. If $F_i$ is a thick level, we will redefine this isotopy between $F_{i-1}$ and $F_{i+1}$ so that $K$ lies on an intermediate surface, making the isotopy into a level isotopy. To do this, use the the bridge compressing disks $B_+$ and $B_-$ for $F_i$ given above so that $F_{i \pm 1}$ is obtained from $F_i$ by compressing along the associated compressing disk $\overline B_{\pm}$. Let $N_+$ and $N_-$ be neighborhoods in $M(\alpha)$ of $B_+$ and $B_-$, respectively. Note that $N_+ \cup N_-$ is a solid torus whose core is $K$, and $F_i$ cuts $\bdy (N_+ \cup N_-)$ into two longitudinal annuli, $A_+$ and $A_-$ (see Figure~\ref{fig:LevelIsotopy}). Moreover, $F_{i \pm 1}$ is isotopic to $(F_i - (N_+ \cup N_-)) \cup A_{\pm}$. We can now define an isotopy between $F_{i-1}$ and $F_{i+1}$ that keeps the surface fixed outside of $N_+ \cup N_-$, and inside $N_+ \cup N_-$ isotopes $A_-$ to $A_+$ (rel $\bdy$) such that $K$ lies flat on an intermediate annulus. Doing this at each thick level thus yields a level isotopy from $F$ to $G$.
\end{proof}

\begin{figure}
\psfrag{N}{$N_+$}
\psfrag{n}{$N_-$}
  \centering
  \includegraphics[width=3in]{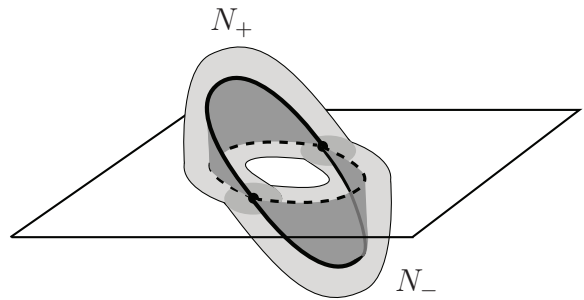}
  \caption{The solid torus $N_+ \cup N_-$}
  \label{fig:LevelIsotopy}
\end{figure}

\bibliographystyle{alpha}
\bibliography{Thinisotopies}

\end{document}